\documentclass[letterpaper, conference, final]{IEEEtran}
\IEEEoverridecommandlockouts
\usepackage{graphics} 
\usepackage{epsfig} 
\usepackage{times} 
\usepackage{amsmath} 
\usepackage{amssymb}  
\usepackage{amsmath} 
\usepackage{amssymb}  

\usepackage{balance}

\usepackage{multirow}
\usepackage{color}
\usepackage{hyperref}

\newtheorem{thm}{Theorem}
\newtheorem{cor}{Corollary}

\newtheorem{defn}{Definition}

\usepackage{algorithm, algorithmicx, algpseudocode}
\newenvironment{proof}{\quad {\it Proof.\,}}{\hfill \IEEEQED \par}
\newenvironment{proof-of}[1]{{\it Proof of #1:\,}}{\hfill\QED\par}
\newtheorem{rem}{Remark}
\newtheorem{example}{Example}

\newcommand{\R}{{\mathbb R}}

\renewcommand{\S}{{\mathbb S}}

\newcommand{\cT}{{\mathcal T}}

\newcommand{\cFW}{{\mathcal{FW}}}

\newcommand{\tr}{\textrm{trace}}

\newcommand{\bfone}{\mathbf{1}}

\begin{document}
\title{Block Factor-Width-Two Matrices in Semidefinite Programming}
\author{Aivar Sootla, Yang Zheng, and Antonis Papachristodoulou
\thanks{The authors are with the Department of Engineering Science, University of Oxford, Parks Road, Oxford, OX1 3PJ, U.K. e-mail: \{aivar.sootla, yang.zheng, antonis\}@eng.ox.ac.uk.  AS and AP are supported by EPSRC Grant EP/M002454/1, YZ is supported in part by the Clarendon Scholarship,
	and in part by the Jason Hu Scholarship. }}
\maketitle

\begin{abstract}
	In this paper, we introduce a set of block factor-width-two matrices, which is a generalisation of factor-width-two matrices and is a subset of positive semidefinite matrices. The set of block factor-width-two matrices is a proper cone and we compute a closed-form expression for its dual cone. We use these cones to build hierarchies of inner and outer approximations of the cone of positive semidefinite matrices. The main feature of these cones is that they enable a decomposition of a large semidefinite constraint into a number of smaller semidefinite constraints. As the main application of these classes of matrices, we envision large-scale semidefinite feasibility optimisation programs including sum-of-squares (SOS) programs. We present numerical examples from SOS optimisation showcasing the properties of this decomposition.
\end{abstract}
\section{Introduction}

Optimisation programs with positive semidefinite (PSD) constraints (or semidefinite programs --- SDPs) are one of the major computational tools in linear systems theory~\cite{boyd2004convex,boyd1994linear}. The introduction of sum-of-squares polynomial optimisation (or SOS programming)~\cite{powers1998algorithm, parrilo2003semidefinite} (and the dual moment approach~\cite{lasserre2010moments}) extended the use of SDPs to polynomial optimisation and thus allowed addressing many nonlinear control problems in polynomial time.

Modern SDPs (and especially SOS programs) are often large-scale, that is, the PSD constraints have large dimensions. Consequently, developing fast SDP solvers has received considerable attention in the literature. Solvers for sparse programs were developed in~\cite{zheng2016fast, zheng2016admmdual, zheng2017chordal} (ADMM-based) and in~\cite{nakata2003exploiting, kim2011exploiting} (interior-point solver) and a general purpose ADMM-based solver was developed in~\cite{ocpb:16}. The sparsity of the PSD constraint was also exploited in the context of SOS programming~\cite{waki2006sums, zheng2018sparse, zheng2018decomposition}. The key idea in these sparsity-exploiting approaches is to decompose large PSD constraints into a number of smaller PSD constraints, while the optimal objective of the program remains the same for a special class of sparsity patterns~\cite{vandenberghe2015chordal}. Since the PSD constraint typically induces the largest computational burden, the computational time can be significantly reduced by using these techniques. These sparsity exploiting techniques can also be used for linear control applications~\cite{chordalZheng2018ecc}.

A related approach to speed-up SOS programming was taken in~\cite{ahmadi2017dsos}, where the authors replaced the PSD cone with the cone of \emph{factor-width-two matrices} (which we denote $\cFW_2^N$ where $N$ stands for the dimension of the matrix). A matrix has a factor width two if it can be represented as a sum of rank two PSD matrices~\cite{boman2005factor} and hence it is also PSD. A certificate for $\cFW_2^N$ matrices can be written as a number of second-order cone constraints, which can reduce the computational and memory burden as demonstrated in~\cite{ahmadi2017dsos}. We note that $\cFW_2^N$ matrices are also scaled diagonally dominant (SDD) as discussed in~\cite{boman2005factor}. The reader unfamiliar with SDD matrices is referred to~\cite{berman1994nonnegative} for details. We only highlight that the individual entries of SDD matrices satisfy a particular set of constraints.

As discussed in~\cite{ahmadi2017dsos} the size of the cone $\cFW_2^N$ is significantly smaller than the size of the PSD cone, therefore, the restricted problem may be infeasible or the optimal solution of the $\cFW_2^N$ program may be significantly different from the optimal solution of the original SDP. There are several approaches to bridge this restriction gap, cf.~\cite{basis_pursuit}. For example, one can employ factor-width-$k$ matrices, which can be decomposed into a sum of PSD matrices of rank $k$. Enforcing this constraint, however, is problematic due to a large number of $k\times k$ PSD constraints, which is $N$ choose $k$, i.e., $N \choose k$. Therefore, the computational burden can actually increase in comparison to the original SDP.

In this paper, we take a different route in order to enrich the cone of factor-width-two matrices: We draw inspiration from SDD matrices and consider their block extension. The key idea of this extension is to partition a matrix into a set of non-intersecting blocks of entries and enforce the SDD constraints on these blocks instead of the individual entries~\cite{feingold1962block}. We introduce the class of \emph{block factor-width-two matrices} based on the block SDD definitions from~\cite{sootla2016existence,sootla2017blocksdd}. A block factor-width-two matrix is also PSD and the constraint ``the matrix is block factor-width-two'' can be enforced using a number of PSD constraints whose size is determined by the size of the blocks. We proceed by deriving a hierarchy of inner and outer approximations of the PSD cone based on the block partition. We propose to use this approximation in SDPs by replacing the PSD cone constraint with a block ``factor-width-two'' constraint. The optimal objective value of the SDPs typically cannot be achieved using this technique, however, the computational cost is reduced. Striking the balance between the accuracy of the solution and the speed can be delicate in general, therefore, we envision the feasibility of SDPs without a specific sparsity structure as the main application. For example, finding a Lyapunov function certifying stability of a nonlinear system often results in a feasibility SDP without a particular sparsity structure. Therefore, in this paper, we mainly focus on SOS programs as an application.

In Section~\ref{s:prel} we cover preliminaries. In Section~\ref{s:block-sdd} we introduce block factor-width-two matrices, a hierarchy of inner and outer approximations of the PSD cone and their SDP and SOS applications. We present numerical examples in Section~\ref{s:examples} and conclude the paper in Section~\ref{s:con}.

\emph{Notation.} The matrix $A^T$ denotes the transpose of $A\in\R^{n\times n}$. We denote the
sets of $n$ by $n$ symmetric, positive definite, positive semidefinite matrices as $\S^n$, $\S_{+}^n$, $\S_{++}^n$, respectively. We use $I_k$ to denote an identity matrix of dimension $k \times k$.

\section{Preliminaries}\label{s:prel}

\subsection{Partitioned Matrices} \label{ss:partitioned-matrices}

We say that a matrix $A\in\R^{N\times N }$ has \emph{$\alpha=\{k_1, \dots, k_p\}$-partition} with $N = \sum\limits_{i = 1}^p k_i$, if $A$ can be written as
\[
A = \begin{bmatrix}
A_{1 1}    & A_{1 2}     & \dots   & A_{1 p} \\
A_{2 1}    & A_{2 2}     & \dots   & A_{2 p} \\
\vdots     & \vdots      & \ddots  & \vdots  \\
A_{p 1}    & A_{p 2}     & \dots   & A_{p p}
\end{bmatrix},
\]
where $A_{i j}\in\R^{k_i\times k_j}$. For a partition $\alpha = \{k_1,\dots, k_p\}$ we define block basis matrices
\begin{equation} \label{eq:blockbasis}
E_{ij} = \begin{bmatrix} E_i^T & E_j^T \end{bmatrix}^T \in \mathbb{R}^{(k_i+k_j)\times N}, i \neq j,
\end{equation}
where
\begin{gather*}
I = \begin{bmatrix} I_{k_1} & & & \\ &  I_{k_2} & &  \\ & & \ddots & \\ & & & I_{k_p}\end{bmatrix} = \begin{bmatrix}E_1 \\ E_2 \\ \vdots \\ E_{p} \end{bmatrix},\\
E_i = \begin{bmatrix} 0 & \ldots & I_{k_i} & \ldots & 0\end{bmatrix} \in \mathbb{R}^{k_i\times N}.
\end{gather*}

We also define a relation between a partition $\beta$ of the matrix $A$ and a coarser partition $\alpha$ of the same matrix.

\begin{defn} Let $\alpha =\{k_1, \dots, k_{p_1}\}$ and $\beta = \{l_1, \dots, l_{p_2} \}$, where $p_1 < p_2$ and $\sum_{i = 1}^{p_1} l_i = \sum_{i = 1}^{p_2} k_i$. We say that $\beta$ is a sub-partition of $\alpha$ and write $\alpha \sqsupseteq \beta$, if there exist integers $\{m_i\}_{i = 1}^{p_1}$ such that $k_i = \sum_{j = m_{i} }^{m_{i+1}-1} l_{j}$ and $m_1 =1$, $m_{p_1} = p_2$, $m_i < m_{i+1}$ for all $i$.
\end{defn}

For example, given $\alpha = \{4,2\}$, $\beta = \{2,2,2\}$ and $\gamma = \{1,1,1,1,1,1\}$, we have $\alpha \sqsupseteq \beta \sqsupseteq \gamma$.

\subsection{Semidefinite and sum-of-squares programming} \label{ss:optim-progs}
The standard \emph{primal-form} semidefinite program (SDP) is an optimisation problem of the form:
\begin{equation}
\begin{aligned}
\min_{X} \quad & \langle C,X \rangle, \\
\text{subject to} \quad & \langle A_i,X \rangle= b_i, i = 1, \ldots, m, \\
& X \in \mathbb{S}^N_+,
\end{aligned} \label{prog:standard-primal}
\end{equation}
where $C, A_i \in \mathbb{S}^N, i = 1, \ldots, m$ and $b \in \mathbb{R}^m$ are given problem data.

SDPs have found many applications in linear systems theory, such as stabilization and $\mathcal{H}_2/\mathcal{H}_{\infty}$ control~\cite{boyd1994linear}. Also, nonlinear control problems in a polynomial field can often be written as polynomial optimisation programs: Given a set of polynomials $f_0(x), f_1(x), \ldots, f_m(x)$ (their coefficients are given) with $x\in\R^n$ and the vector $b \in \mathbb{R}^m$ we aim to solve
\begin{equation}
\begin{aligned}
\min_{y} \quad & \langle b, y \rangle, \\
\text{subject to}  \quad & f_0(x) + \sum_{i=1}^m y_i f_i(x) \ge 0,\,\,\forall x\in\R^n.
\end{aligned} \label{prog:pop}
\end{equation}

Even though the nonnegativity constraint in~\eqref{prog:pop} is convex, the program is infinite dimensional due to the dependence on $x$. Therefore, a tractable sum-of-squares relaxation of the nonnegative constraint is typically used. Given $x \in \mathbb{R}^n$, a polynomial $p(x)$ of degree $2d$ is called a sum-of-squares (SOS) polynomial if it can be written into a sum of squares of other polynomials of degree no greater than $d$. It is known (\cite{papachristodoulou2005tutorial,ParriloPhD}) that $p(x)$ admits an SOS decomposition if and only if there exists $Q \in \mathbb{S}^N_+$ with $N = {n + d \choose d}$ such that
\begin{equation} \label{eq:SOS}
p(x) = v_d(x)^T Q v_d(x),
\end{equation}
where $v_d(x)$ is a vector of monomials of degree no greater than $d$. Replacing the nonnegative constraint with an SOS constraint yields the following optimisation program:
\begin{equation}
\begin{aligned}
\min_{y}\quad& \langle b, y \rangle, \\
\text{subject to} \quad& v_d(x)^T Q v_d(x) = \\
                       & \qquad\qquad f_0(x) + \sum_{i=1}^m y_i f_i(x),\,\forall x\in\R^n,\\
&Q\in\mathbb{S}_+^N,
\end{aligned} \label{prog:sos}
\end{equation}
where the constraints imply that $f_0(x) + \sum_{i=1}^m y_i f_i(x)$ is an SOS polynomial. Matching the coefficients on both sides polynomial equality leads to a set of linear equality constraints on $Q$ and $y$, and we obtain an SDP of the form~\eqref{prog:standard-primal} with additional free variables.

\subsection{Factor Width of Positive Semidefinite Matrices}

It is well-known that small and medium-sized SDPs can be solved up to an arbitrary accuracy in polynomial time via interior point methods~\cite{boyd2004convex}. However, as the size of the PSD cone $N$ in~\eqref{prog:standard-primal} increases, the current state-of-the-art interior point algorithms become impractical in terms of memory requirements, computational burden or numerical accuracy. In the recent work~\cite{ahmadi2017dsos} it was proposed to speed up semidefinite and SOS optimisation by replacing the PSD cone by a cone of factor-width-two matrices. This work is based on the following definitions from~\cite{boman2005factor}.

\begin{defn}
	A matrix $X\in\S_+^{N}$ belongs to the class of factor-width-$k$ matrices (denoted as $\cFW_k^N$) if and only if
	\begin{gather*}
	X = \sum\limits_{i = 1}^s e_i^T X_i e_i, \text{  with  } X_i \in \S^k_+,\, e_i \in \cT_k,
	\end{gather*}
	where $\cT_k$ is a collection of matrices $e_i\in\R^{k \times N}$ with every row having only one non-zero element equal to one, the columns being orthonormal, and
	$s  = {N \choose k}$.
\end{defn}

The matrices $e_{i}$ can be seen as a decomposition basis for the matrix $X$. It can be shown that a dual (with respect to the trace inner product) set to $\cFW_k^N$ can be characterised as follows
\begin{gather*}
(\cFW_k^N)^\ast = \{Z\in \S^N | e_{i} Z e_{i}^T \in \S_{+}^{k},\,\, \forall e_i \in \cT_k\}.
\end{gather*}
One can also show that the following hierarchy of inner  and outer  approximations of $\S_+^N$ holds:
\begin{equation*}
\begin{aligned}
\cFW_1^N \subset &\cFW_{2}^N \subset \ldots \subset \cFW_N^N = \S_+^N = \\
&(\cFW_N^N)^\ast \subset \ldots \subset(\cFW_{2}^N)^\ast \subset (\cFW_1^N)^\ast.
\end{aligned}
\end{equation*}

Replacing $\mathbb{S}^N_+$ in~\eqref{prog:standard-primal} with $\cFW_{k}^N$ leads to a restriction with multiple $k \times k$ PSD cones. In particular, the factor-width-two matrices can be written as
\begin{gather*}
X = \sum\limits_{i = 1}^{N-1} \sum\limits_{j=i+1}^N E_{i j}^T X_{i j} E_{i j}, \text{  with  } X_{i j} \in \S^2_+,
\end{gather*}
where the matrices $E_{i j}$ are defined as in Section~\ref{ss:partitioned-matrices} with $\alpha = \{1, \dots, 1\}$ and $p = N$. The constraints $X_{i j}\in\S^2_+$ can be equivalently written as second-order cone constraints, which can be solved much faster compared to solving SDPs. This fact has been used to solve large-scale SOS programs in~\cite{ahmadi2017dsos}. However, the solution from $\cFW_2^N$ might be very conservative. As was pointed out in~\cite{ahmadi2017dsos}, increasing the factor width may reduce the degree of conservatism, but this requires working with a combinatorial number ${N \choose k}$ of PSD cones of size $k \times k$, which is not practical.

\section{Block Factor-width-two Matrices}\label{s:block-sdd}

\subsection{Block factor-width-two matrices}
In this section, we introduce the class of block factor-width-two matrices, which is less conservative than $\cFW_2^N$ and more scalable than $\cFW_k^N$ ($k \geq 3$).
\begin{defn}
	We say that $\alpha =\{k_1, \dots, k_p\}$-partitioned matrix $X\in\S^N$ belongs to the class $\cFW^N_{\alpha, 2}$ if and only if
	\begin{gather*}
	X = \sum\limits_{i = 1}^{p-1}\sum\limits_{j = i+1}^p E_{i j}^T X_{i j} E_{i j},
	\end{gather*}
	where $X_{i j} \in \S^{k_i + k_j}_+$ and $E_{ij}$ are defined in~\eqref{eq:blockbasis}.
\end{defn}

It is straightforward to show that the set $\cFW_{\alpha,2}^N$ is a cone with a non-empty interior, which is also:
\begin{itemize}
	\item{\it convex:} for any $X, Y\in\cFW_{\alpha,2}^N$, $0\leq \theta \leq 1$, we have that $\theta X + (1 - \theta) Y\in\cFW_{\alpha,2}^N$,
	\item{\it salient:} for any nonzero $X\in\cFW_{\alpha,2}^N$, $-X\not \in\cFW_{\alpha,2}^N$,
	\item{\it pointed:} the zero matrix is in $\cFW_{\alpha,2}^N$.
\end{itemize}
We will show in what follows that this cone is additionally closed, which makes it a proper cone (closed, convex, pointed, salient cone with non-empty interior).

The main difference with the definition of factor-width-two matrices comes in the partition $\alpha$, which dictates the sizes of $X_{i j}$'s and $E_{i j}$'s, as well as their number. The number of basis matrices $E_{i j}$ is the same as in the case when we treat every block $X_{i j}$ as a scalar and apply the factor width decomposition to it. In our definition, we have a fixed partition $\alpha$ and a fixed ``block factor-width'', which is equal to two. In order to create a hierarchy of approximations of $\S^N_+$ we can increase the ``block factor-width'', which means increasing the number of basis matrices $E_{i j}$. However, we can also build a hierarchy based on the partition coarsening, which reduces the number of basis matrices $E_{i j}$.

\begin{thm}\label{prop:fw-alpha-k}
	Given  $\alpha=\{k_1,\dots, k_p\}$, $\beta = \{\widetilde k_1, \dots, \widetilde k_q \}$ and $\alpha\sqsupseteq \beta$, we have the following inclusion:
	\begin{multline*}
	\cFW_2^N = \cFW_{\bfone,2}^N \subset \cFW_{\beta, 2}^N \subset \cFW_{\alpha, 2}^N\subset \\
	\cFW_{\bfone,\max_{i \ne j}\{k_i + k_j\}}^N \subset \cFW_{\{K_1, K_2\}, 2}^{N} = \S_+^N,
	\end{multline*}
	where $\bfone = \{1,1,\ldots, 1\}$, $K_1$, $K_2$ are positive integers and $K_1+K_2 = N$.
\end{thm}
\begin{proof}
	First, $\cFW_2^N = \cFW_{\bfone,2}^N$, $\cFW_{\max_{i \ne j}\{k_i + k_j\}} \subset \cFW_{\{K_1, K_2\}, 2}^{N} = \S_+^N $ hold by definition. Furthermore, $\cFW_{\alpha, 2}^N\subset \cFW_{\bfone, \max_{i \ne j}\{k_i + k_j\}}^N$ is true since in the decomposition for $\cFW_{\alpha, 2}^N$ we use PSD matrices of dimension at most $\max_{i \ne j}\{k_i + k_j\}$.
	
	In order to prove $\cFW_{\beta, 2}^N \subset \cFW_{\alpha, 2}^N$ it suffices to consider the case $\beta = \{k_1, \dots,k_{p-1}, k_p, k_{p+1} \}$, $\alpha = \{k_1, \dots, k_{p-1}, k_{p} + k_{p+1}\}$. Let $E_{\beta i j}$ for $i,j = 1, \dots, p+1$ be the decomposition basis for the $\beta$-partition and $E_{\alpha i j}$ for $i,j = 1, \dots, p$ be the decomposition basis the $\alpha$-partition. By the premise, there exist $X_{ i j}\in\S_+^{k_i+k_j}$ such that:
	\begin{multline*}
	X = \sum\limits_{i = 1}^{p}\sum\limits_{j = i+1}^{p+1} E_{\beta i j}^T X_{ i j} E_{\beta i j} = \sum\limits_{i = 1}^{p-1}\sum\limits_{j = i+1}^{p-1} E_{\beta i j}^T X_{i j} E_{\beta i j} +\\ \sum\limits_{i = 1}^{p-1} E_{\beta i p}^T X_{i p} E_{\beta i p} +\sum\limits_{i = 1}^{p} E_{\beta i (p+1)}^T X_{i (p+1)} E_{\beta i (p+1)}.
	\end{multline*}
	We need to construct $\widetilde X_{ i j}$ so that $X$ is decomposed as:
	\begin{align}
	X=\sum\limits_{i = 1}^{p-1}\sum\limits_{j = i+1}^{p} E_{\alpha i j}^T \widetilde X_{ i j} E_{\alpha i j}. \label{decomp_alpha}
	\end{align}
	Since the first $p-1$ blocks in both partitions are the same, we have that $E_{\alpha i j} = E_{\beta i j}$ and $\widetilde X_{i j} = X_{i j}$ for all $i, j <p$. Therefore, in order to obtain the decomposition~\eqref{decomp_alpha}, it remains to construct $\widetilde X_{i p}$ for $i =  1,\dots, p-1$ such that
	\begin{multline}\label{decom_ident}
	\sum\limits_{i = 1}^{p-1} E_{\alpha i p}^T \widetilde X_{i p} E_{\alpha i p} = \sum\limits_{i = 1}^{p-1} E_{\beta i p}^T X_{i p} E_{\beta i p} +\\\sum\limits_{i = 1}^{p} E_{\beta i (p+1)}^T X_{i (p+1)} E_{\beta i (p+1)}.
	\end{multline}
	Consider the matrices $X_{i j}$ for $i < p$ and $j = p, p+1$ and split them according to the partition
	\begin{gather*}
	X_{i j} = \begin{pmatrix}
	X_{i j}^{1 1} &     X_{i j}^{1 2} \\
	X_{i j}^{1 2}  &     X_{i j}^{2 2}
	\end{pmatrix},
	\end{gather*}
	where $X_{i j}^{1 1}\in \S_+^{k_i}$, $X_{i j}^{1 2}\in \R^{k_i\times k_j}$, $X_{i j}^{2 2}\in \S_+^{k_j}$.
	It can be verified  by direct computation that the identity~\eqref{decom_ident} holds if $\widetilde X_{i p}$ for $i<p$ are chosen as follows:
	\begin{multline*}
	\widetilde X_{i p} = \begin{pmatrix}
	0& 0 & 0\\
	0 &  {X_{p (p+1)}^{1 1}} & {X_{p (p+1)}^{1 2}} \\
	0 & {X_{p (p+1)}^{1 2}} & {X_{p (p+1)}^{2 2}}\end{pmatrix}\frac{1}{p-1}
	+\\\begin{pmatrix}
	X_{i (p+1)}^{1 1}& 0& X_{i (p+1)}^{1 2}\\
	0 & 0 & 0 \\
	X_{i (p+1)}^{1 2} & 0 & X_{i (p+1)}^{2 2}
	\end{pmatrix}+    \begin{pmatrix}
	X_{i p}^{1 1} & X_{i p}^{1 2} & 0\\
	X_{i p}^{1 2} & X_{i p}^{2 2} & 0\\
	0 & 0 & 0
	\end{pmatrix}.
	\end{multline*}
	Thus we complete the proof.
\end{proof}

\begin{example} Consider the following PSD matrix
	\begin{gather*}
	X = \begin{pmatrix}
	6  &   8  &  -2   & -2\\
	8  &  16  &   1   &  1\\
	-2 &    1 &   10  &  -1\\
	-2 &    1 &   -1  &  24
	\end{pmatrix}.
	\end{gather*}
	It can be verified that $X\in\cFW_{\beta,2}^4$ for the partition $\beta = \{1, 1, 1,1\}$ and the matrices in the decomposition can be chosen as follows:
	\begin{gather*}
	X_{1 2} =\begin{pmatrix}
	4.5 &   8\\
	8   & 14.5
	\end{pmatrix}, X_{1 3} = \begin{pmatrix}
	1  &-2\\
	-2 &   6
	\end{pmatrix}, X_{2 3} = \begin{pmatrix}
	1 &  1\\
	1 &  2
	\end{pmatrix},\\
	X_{1 4} =\begin{pmatrix}
	0.5 &  -2\\
	-2  & 12
	\end{pmatrix}, X_{2 4}= \begin{pmatrix}
	0.5 &  1\\\
	1   & 6
	\end{pmatrix}, X_{3 4} =\begin{pmatrix}
	2   &-1\\
	-1  & 6
	\end{pmatrix}.
	\end{gather*}
	If we collapse the last two entries into a block and obtain the partition $\alpha = \{1, 1, 2\}$, then we can use the constructions in Theorem~\ref{prop:fw-alpha-k} in order to obtain the matrices $\widetilde X_{1 2}= X_{1 2}$,
	\begin{gather*}
	\widetilde X_{1 3} = \begin{pmatrix}
	1.5&   -2&   -2\\
	-2 &    7&-0.5\\
	-2 & -0.5& 15
	\end{pmatrix}, \widetilde X_{2 3} = \begin{pmatrix}
	1.5&   1  &  1\\
	1  &    3,&-0.5\\
	1  & -0.5 & 9 \end{pmatrix}.
	\end{gather*}
	The matrices $\widetilde X_{1 2}$, $\widetilde X_{1 3}$, $\widetilde X_{2 3}$ are PSD, which shows that $X\in\cFW_{\alpha, 2}^4$.
\end{example}

We can also describe a dual set of matrices to $\cFW_{\alpha, 2}^N$ matrices (with respect to the trace inner product), which creates an outer approximation hierarchy for the cone $\S^N_+$.
\begin{cor}
	The dual to $\cFW_{\alpha, 2}^N$ with respect to the trace inner product is defined as:
	\begin{gather*}
	(\cFW_{\alpha, 2}^N)^\ast = \{Z\in \S^N | E_{i j} Z E_{i j}^T \in \S_{+}^{k_i+k_j},\,\, \forall 1\le i<j \le p\}.
	\end{gather*}
	Furthermore, let  $\alpha=\{k_1,\dots, k_p\}$ and $\beta = \{\widetilde k_1, \dots, \widetilde k_q \}$,  $\alpha\sqsupseteq \beta$, then we have the following inclusions:
	\begin{multline*}
	(\cFW_{\bfone, 2}^N)^\ast \supset (\cFW_{\beta, 2}^N)^\ast \supset (\cFW_{\alpha, 2}^N)^\ast\supset \\
	(\cFW_{\bfone, \max_{i \ne j}\{k_i + k_j\}}^N)^\ast \supset (\cFW_{\{K_1, K_2,\}, 2}^N)^\ast = \S_+^N,
	\end{multline*}
	where $K_1$, $K_2$ are positive integer and $K_1+K_2 = N$.
\end{cor}
\begin{proof}
	The proof of the first part follows after noticing that for any matrix $Z\in\S^N$ such that $E_{i j} Z E_{i j}^T \in \S_{+}^{k_i}$ for all $1\le i<j \le p$, and for any matrix $X \in\cFW_{\alpha,2}^N$, we have $\tr(X Z) \ge 0$. The proof of the second part is straightforward.
\end{proof}
\begin{rem}
	Using the terminology in~\cite{sun2014decomposition} the cone $(\cFW_{\alpha, 2}^N)^\ast$ is a partially separable cone, which ensures that its dual is $\cFW_{\alpha, 2}^N$ and $(\cFW_{\alpha, 2}^N)^\ast$ is a proper cone.
\end{rem}

The major difference between our hierarchy of $\cFW_{\alpha,2}^N$ and the hierarchy of $\cFW_k^N$ is the number of basis matrices, which in our case is substantially lower due to two reasons: We use factor-width-two generalisations and we coarsen the partitions. Therefore the number of basis matrices is equal to $p(p-1)/2$ for $\alpha = \{k_1,\dots, k_p\}$, and as we make a partition coarser the number $p$ and hence the number of basis matrices decreases. We note however that the set $\cFW_3^N$ is not contained in $\cFW_{\alpha, 2}^N$ for $p>2$. This is because $\cFW_3^N$ contains \emph{all possible combinations of} $e_i$'s as the basis vectors (read all possible partitions). In contrast, $\cFW_{\alpha, 2}^N$ will not consider certain choices of partitions. Therefore, our approach has a particular advantage in applications where partitions come as a natural property of the problem.

\subsection{Applications to SDPs in the Standard Primal Form} \label{section:ApplicationsSDP}
The main idea is to replace the cone $\S_+^N$ with $\cFW_{\alpha,2}^N$ or its dual in order to obtain a restriction of the original program. Consider a restriction of~\eqref{prog:standard-primal}, where we assume the matrix $X$ is partitioned according to $\alpha = \{k_1, \dots, k_p \}$,
\begin{equation}
\begin{aligned}
\min_{X} \quad & \langle C,X \rangle, \\
\text{subject to} \quad & \langle A_i,X \rangle= b_i, i = 1, \ldots, m \\
& X \in \cFW_{\alpha,2}^N.
\end{aligned} \label{prog:standard-restr_s1}
\end{equation}

This program can be cast in the SDP form as follows:
\begin{equation}
\begin{aligned}
\min_{X_{lj}} \quad          & \sum_{j = 1}^{p-1}\sum_{l = j+1}^p \langle E_{l j} C E_{l j}^T,X_{l j}  \rangle, \\
\text{subject to} \quad & \sum_{j = 1}^{p-1}\sum_{l = j+1}^p \langle E_{l j} A_i E_{l j}^T, X_{l j} \rangle= b_i, i = 1, \ldots, m \\
                        & X_{l j} \in  \S_+^{k_l + k_j},\,\, 1 \le j < l\le p.
\end{aligned} \label{prog:primal-restr}
\end{equation}
which is amenable for a straightforward implementation in standard SDP solvers such as SeDuMi~\cite{Sedumi}, MOSEK~\cite{andersen2000mosek} or SCS~\cite{ocpb:16}. This program has the same number of equality constraints as~\eqref{prog:standard-primal}, but the number and the dimensions of PSD constraints are different. We can also perform a relaxation of~\eqref{prog:standard-primal} by replacing $X\in\S_+^N$ by $X\in(\cFW_{\alpha, 2}^N)^\ast$. We will not discuss the relaxation in detail since we focus on the restriction of the primal SDP.

\section{Numerical Examples} \label{s:examples}
In our numerical examples, we used YALMIP~\cite{YALMIP} in order to reformulate the polynomial optimisation program into a standard SDP and we solve the SDPs using MOSEK~\cite{andersen2000mosek}\footnote{Code is available via \url{https://github.com/zhengy09/SDPfw}.}
\subsection{Polynomial Optimisation} \label{ss:unstructured}
We consider the polynomial optimisation problem:
\begin{equation} \label{Eq:ex1}
\begin{aligned}
\min_{\gamma} \quad & -\gamma \\
\text{subject to} \quad & q(x) - \gamma \geq 0,\,\,\, \forall x\in \R^n,
\end{aligned}
\end{equation}
where
\begin{multline*}
q(x) = ((3-2 x_1) x_1-2 x_2+1)^2 + \\
+\sum\limits_{i =2}^{n-1} ((3 -2 x_i) x_i - x_{i-1} - 2 x_{i +1} +1)^2 + \\
+((3-2 x_n) x_n - x_{n-1} + 1)^2 + \left(\sum\limits_{i = 1}^n x_i\right)^2.
\end{multline*}
We added the last term, so that the problem does not enjoy the structure exploited by the methods in~\cite{waki2006sums, zheng2018sparse}. We vary $n$ and obtain different semidefinite optimisation problems in the standard primal form with constraints of different sizes listed in Table~\ref{tab:pop_time}.

\begin{table} \caption{Computational Results for an SDP Partition in Section~\ref{ss:unstructured}}\label{tab:pop_time}
	\centering
	\begin{tabular}{l|c|cccc}
		\multirow{2}{10pt}{$n$}   & \multirow{2}{30pt}{Full SDP} & \multicolumn{4}{c}{Number of  Partitions in SDP Variables} \\
		&                              & $4$             &  $10$              &  $20$            & $50$        \\
		\hline
		\hline
		\multicolumn{6}{c}{{Computational Time (seconds)}} \\
		\hline
		$10$                    &  $2.38$                      & $1.43$           & $1.29$            & $1.28$           & $1.49$      \\
		$15$                    &  $27.3$                      & $23.3$           & $15.6$            & $10.1$           & $5.36$      \\
		$20$                    &  $489$                       & $252$            &  $98.1$           & $66.8$           & $28.1$      \\
		$25$                    &  $\infty$                    & $1.97\cdot 10^3$ &  $7.83\cdot 10^2$ & $5.71\cdot 10^2$ & $1.32\cdot 10^2$\\
		$30$                    &  $\infty$                    & $\infty$         &  $5.68\cdot 10^3$ & $3.71\cdot 10^3$ & $8.4\cdot 10^2$\\
		\hline
		\multicolumn{6}{c}{{Objective values}} \\
		\hline
		$10$                      &   $-0.9$                    & $-0.45$ &  $134$ &  $483$ & $2.12\cdot 10^3$ \\
		$15$                      &   $-0.92$                   & $-0.75$ & $80.1$ &  $459$ & $2.24\cdot 10^3$ \\
		$20$                      &   $-0.87$                   & $-0.87$ &$-0.11$ &  $251$ & $1.91\cdot 10^3$ \\
		$25$                      &   $\infty$                  & $-1.07$ &$-0.21$ &  $231$ & $1.36\cdot 10^3$ \\
		$30$                      &   $\infty$                  &$\infty$ &$-0.37$ &  $177$ & $1.77\cdot 10^3$ \\
		\hline
		\multicolumn{6}{c}{Sizes of SDP Constraints} \\
		\hline
		$10$                      &   $66$                    & $32-34$ &  $12-14$ &  $6-8$ & $2-4$ \\
		$15$                      &   $136$                   & $68$ & $26-28$ &  $12-14$ & $4-6$ \\
		$20$                      &   $231$                   & $114-116$ &$46-48$ &  $22-24$ & $8-10$ \\
		$25$                      &   $351$                  & $174-176$ &$70-72$ &  $34-36$ & $14-16$ \\
		$30$                      &   $496$                  &$248$ &$98-100$ &  $48-50$ & $18-20$
	\end{tabular}
\end{table}

We partition the SDP as discussed in Section~\ref{section:ApplicationsSDP}. We fix the partition size $p$ and we choose the size of the blocks as the closest integers to $N/p$, where $N$ is the size of the SDP constraint. In particular, if $k_1\le N/n\le k_2$, then we pick the maximum number of blocks of size $k_1$ and the rest of size $k_2$. The number of SDP constraints, as discussed above is equal to $p(p-1)/2$. Note that the number of linear constraints remains the same as in the full SDP.

We present the computational times and the objective values in Table~\ref{tab:pop_time}. It noticeable that with a finer partition we obtain faster solutions, which are, however, conservative in terms of the objective function. Fine partitions may be very useful for feasibility programs, while coarse partitions are competitive for large SDPs, where the value of the objective function is important. Note that for large-scale instances $n \geq 25$, MOSEK ran out of memory on our machine. On the other hand, our strategy of using block factor-with-two matrices can still provide a useful upper bound for~\eqref{Eq:ex1}.

\subsection{Matrix Sum-of-Squares Programming\label{ss:maxtrix_sos}}
In our second example, we show that there exists a natural partition $\alpha$ in the case of the matrix-vision of SOS programs. Indeed, consider a polynomial matrix constraint:
\begin{gather*}
P(x) = \begin{bmatrix}
p_{1 1}(x) & p_{12}(x) & \dots  & p_{1 n}(x) \\
p_{2 1}(x) & p_{22}(x) & \dots  & p_{2 n}(x) \\
\vdots     & \vdots    & \ddots & \vdots \\
p_{n 1}(x) & p_{n 2}(x) & \dots  & p_{n n}(x)
\end{bmatrix}\succeq 0,\,\, \forall x\in\R^m.
\end{gather*}
Treating this constraint directly is intractable and the usual technique is the SOS-relaxation, which results in the following reformulation~\cite{gatermann2004symmetry}
\begin{align}
\label{sos-lin-con} P(x) &= (I_n\otimes v_d(x))^T Q (I_n\otimes v_d(x)), \,\, \forall x\in\R^m\\
\label{sos-sdp-con}Q &\succeq 0,
\end{align}
where the constraint~\eqref{sos-lin-con} is actually a linear constraint linking the coefficients of $P(x)$ with the matrix $Q$. The SOS programs are known to suffer from the curse of dimensionality, in particular, the size of $Q$ grows combinatorially when we vary both $m$ and $d$. Therefore, another technique was proposed in~\cite{ahmadi2017dsos}, which replaces the constraint~\eqref{sos-sdp-con} with
\begin{align}
\label{sdsos-lin-con} P(x) &= (I_n\otimes v_d(x))^T Q (I_n\otimes v_d(x)), \,\, \forall x\in\R^m\\
\label{sdsos-sdp-con}Q &\in\cFW_2^N.
\end{align}
Now instead of the large semidefinite constraint we are dealing with a large number of $2\times 2$ PSD constraints, which can actually be cast as second order cone constraints.

In addition, we can address this problem by replacing the constraint~\eqref{sos-sdp-con} with $Q\in\cFW_{\alpha, 2}^N$ with a natural choice of $\alpha$. In particular, we assume that $P(x)\in \cFW_2^n$ for every $x\in\R^n$ resulting in the decomposition
\begin{gather*}
P(x) = \sum_{i = 1}^{p-1}\sum_{j = i+1}^p E_{i j}^T P_{i j} (x) E_{i j},\,\,\, P_{i j}(x)\succeq 0,
\end{gather*}
and only then use the SOS relaxation on the polynomial matrices $P_{i j}(x)$ of the dimension $2\times 2$. In order to avoid the question of existence of such decompositions, we restrict the search of $P(x)$ to the following set of constraints:
\begin{gather*}
P(x) = \sum_{i = 1}^{p-1}\sum_{j = i+1}^p E_{i j}^T P_{i j} (x) E_{i j},\,\,\, P_{i j}(x) \text{ is SOS}.
\end{gather*}
Some rudimentary linear algebra results in the following reformulation of the PSD constraints:
\begin{align}
\label{block-sdsos-lin-con} P(x) &= (I_n\otimes v_d(x))^T Q (I_n\otimes v_d(x)), \,\, \forall x\in\R^m,\\
\label{block-sdsos-sdp-con}Q &\in\cFW_{\alpha, 2}^N,
\end{align}
where $\alpha$ is pre-determined.

Using the $\cFW_{\alpha, 2}^N$ restriction provides with a larger set of solutions than $\cFW_{\bfone, 2}^N$. For example, the polynomial matrix:
{\small \begin{multline*}
	P(x)= \\\begin{bmatrix}
	4 x^2 + 9 y^2 +0.315 &  x+y                 & x+y \\
	x+y                  & 9 x^2 + 4 y^2 +0.315 & x+y \\
	x+y                  & x+y                  & x^2 + 25 y^2 +0.315
	\end{bmatrix}
	\end{multline*}}
satisfies the constraints~(\ref{block-sdsos-lin-con}, \ref{block-sdsos-sdp-con}), but does not satisfy the constraints~~(\ref{sdsos-lin-con}, \ref{sdsos-sdp-con}).

We further test our approach on the following program:
\begin{gather}
\begin{aligned}
\max_{P(x), \gamma}~~~& \gamma \\
& P(x) - \gamma I  \in \S_+^n,~~\forall x\in\R^3,
\end{aligned}\label{prog:sos-matrix}
\end{gather}
where every entry of $P(x)$ is a random polynomial of degree two in three variables, and we vary the dimension of $P(x)$. The computational results are depicted in Table~\ref{tab:mat_sos}. Our restriction offers faster computational solutions with almost the same optimal objectives compared to the standard SOS technique, while the technique~\cite{ahmadi2017dsos} provides even faster solutions, but their quality is worse.

\begin{table}\caption{Computational results for Section~\ref{ss:maxtrix_sos}} \label{tab:mat_sos}
	\centering
	\begin{tabular}{c|ccccccc}
		\multicolumn{8}{c}{Computational time} \\
		$n$  &  $20$ &  $25$ &  $30$ &  $35$ &  $40$ &  $45$  &  $50$ \\
		\hline
		\hline
		SOS               & $5.28$ & $14.4$& $35.9$& $87.2$& $175.0$& $316.0$& $487.8$ \\
		$\cFW_{\alpha,2}$ & $7.90$& $10.8$& $16.6$ & $25.3$ & $36.0$ & $57.4$ & $71.4$  \\
		$\cFW_{2}$     & $1.04$& $1.1$& $1.3$ & $1.6$ & $2.1$ & $2.6$ & $3.3$  \\
		\multicolumn{8}{c}{Objective value} \\
		\hline
		\hline
		SOS               & $149.0$ & $266.5$& $316.2$& $460.8$& $562.0$& $746.9$& $919.8$ \\
		$\cFW_{\alpha,2}$ & $149.0$ & $266.5$& $316.2$& $460.8$& $562.0$& $746.9$& $919.8$ \\
		$\cFW_{2}$        & $154.4$ & $270.3$& $324.8$& $477.7$& $570.9$& $762.2$& $961.7$
	\end{tabular}
\end{table}

\section{Conclusion and Discussion} \label{s:con}
We introduced a novel class of matrices and presented a hierarchy of inner and outer approximations of the cone of positive semidefinite (PSD) matrices. Both inner and outer approximations are proper cones and enjoy useful duality relations, furthermore, the inclusion certificates for these cones is a set of PSD constraints smaller than the size of the matrix. This allows deriving a hierarchy of scalable relaxations and restrictions of semidefinite programs (SDPs). The inner approximations (cones $\cFW_{\alpha, 2}^N$) are built by partitioning the matrix into a non-intersecting set of entries.
It is not entirely clear at the moment how to build ``the best'' partition in terms of the solution of the particular SDP. However, in some problems, the partition comes naturally from the problem formulation, \emph{e.g.}, the matrix-version of SOS programs discussed in Section~\ref{ss:maxtrix_sos}. Our numerical experiments suggest that these hierarchies can be used for dense large-scale SDPs, which arise in SOS programming.

Our future work will investigate the consequences of block factor-width-two matrices in relevant control applications that involve SDPs. Also, it would be interesting to incorporate the properties of block factor-width-two matrices in the development of first-order algorithms (\emph{e.g.}, the solvers~\cite{ocpb:16,zheng2016fast}) for solving general SDPs.
\balance

\end{document}